\input ./style/arxiv-vmsta.cfg
\documentclass[numbers,compress,v1.0.1]{vmsta}

\volume{3}
\issue{1}
\pubyear{2016}
\firstpage{1}
\lastpage{17}
\doi{10.15559/16-VMSTA48}

\newtheorem{thm}{Theorem}[section]

\newtheorem{lem}{Lemma}[section]

\theoremstyle{definition}
\newtheorem{zau}{Remark}[section]

\DeclareMathOperator{\pr}{\mathsf P}

\DeclareMathOperator{\M}{\mathsf E}

\startlocaldefs

\newcommand{\rrvert}{\vert}
\newcommand{\llvert}{\vert}
\allowdisplaybreaks

\newcommand{\si}{\sigma}
\newcommand{\R}{\mathbb{R}}
\newcommand{\G}{\mathbb{G}}
\newcommand{\T}{\mathbb{T}}
\newcommand{\I}{\mathbb{I}}
\newcommand{\F}{\mathcal{F}}

\newcommand{\Om}{\varOmega}

\endlocaldefs

\begin{document}
\begin{frontmatter}

\title{Functional limit theorems for additive and multiplicative schemes in the Cox--Ingersoll--Ross model}

\author{\inits{Yu.}\fnm{Yuliia}\snm{Mishura}}\email{myus@univ.kiev.ua}
\author{\inits{Ye.}\fnm{Yevheniia}\snm{Munchak}\corref{cor1}}\email{yevheniamunchak@gmail.com}
\cortext[cor1]{Corresponding author.}
\address{Taras Shevchenko National University of Kyiv, Volodymyrska str. 64, 01601, Kyiv, Ukraine}

\markboth{Yu. Mishura, Ye. Munchak}{Functional limit theorems in the Cox--Ingersoll--Ross model}

\begin{abstract}
In this paper, we consider the Cox--Ingersoll--Ross (CIR) process in the
regime where the process does not hit zero. We construct additive and
multiplicative discrete approximation schemes for the price of
asset that is modeled by the CIR process and geometric CIR process. In
order to construct these schemes, we take the Euler approximations of
the CIR process itself but replace the increments of the Wiener
process with iid bounded vanishing symmetric random variables. We
introduce a ``truncated'' CIR process and apply it to prove the weak
convergence of asset prices. We establish the fact that this
``truncated'' process does not hit zero under the same condition
considered for the original nontruncated process.
\end{abstract}

\begin{keyword}
Cox--Ingersoll--Ross process\sep
discrete approximation scheme\sep
functional limit theorems
\MSC[2010]  60F99\sep 60G07\sep 91B25
\end{keyword}

\received{15 December 2015}
%
\revised{22 February 2016}
%
\accepted{25 February 2016}
\publishedonline{3 March 2016}
\end{frontmatter}

\section{Introduction}\label{}
The problem of convergence of discrete-time financial models to the
models with continuous time is well developed; see, e.g., \cite{Broad,Ch,Cutl,delbaen,duff,HZ,Hu-Sc}. The reason for such an interest
can be explained as follows: from the analytical point of view, it is
much simpler to deal with continuous-time models although all
real-world models operate in the discrete time. In what concerns the
rate of convergence, there can be different approaches to its
estimation. Some of this approaches are established in \cite{Broad,Ch,Mishura3,Mishura1,Mishura,Mish_Mun_2,Mish_Mun_3}. In this paper, we
consider the Cox--Ingersoll--Ross process and its approximation on a
finite time interval. The CIR process was originally proposed by Cox,
Ingersoll, and Ross \cite{Cox_Ing_Ross} as a model for short-term
interest rates. Nowadays, this model is widely used in financial
modeling, for example, as the volatility process in the Heston model
\cite{Heston}. The strong global approximation of CIR process is
studied in several articles. Strong convergence (without a rate or with
a logarithmic rate) of several discretization schemes is shown by \cite
{Alfonsi,Bossy,delbaen1,Gyongy,Higham}. In \cite{Alfonsi},
a~general framework for the analysis of strong approximation of the CIR
process is presented along with extensive simulation studies.
Nonlogarithmic convergence rates are obtained in \cite{Berkaoui}. In
\cite{Deelstra}, the author extends the CIR model of the short interest
rate by assuming a stochastic reversion level, which better reflects
the time
dependence caused by the cyclical nature of the economy or by
expectations concerning the future impact of monetary policies. In this
framework, the convergence of the long-term return by using the theory
of generalized Bessel-square processes is studied. In \cite{Pearson},
the authors propose an empirical method that utilizes the conditional
density of the state variables to estimate and test a term structure
model with known price formula using data on both discount and coupon
bonds. The method is applied to an extension of a two-factor model due
to Cox, Ingersoll, and Ross. Their results show that estimates based
solely on bills imply unreasonably large price errors for longer
maturities. The process is also discussed in \cite{Brigo}.

In this article, we focus on the regime where the CIR process does not
hit zero and study weak approximation of this process. In the first
case, the sequence of prelimit markets is modeled as the sequence of
the discrete-time additive stochastic processes, whereas in the second
case, the sequence of multiplicative stochastic processes is modeled.
The additive scheme is widely used, for example, in the papers \cite
{Alfonsi,Bossy,Andreas}. The papers \cite{Deelstra,Pearson} are
recent examples of modeling a stochastic interest rate by the
multiplicative model of CIR process. In \cite{Deelstra}, the authors
say that the model has the ``strong convergence property,'' whereas
they refer to models as having the ``weak convergence property'' when
the returns converge to a constant, which generally depends upon the
current economic environment and that may change in a stochastic
fashion over time. We construct a discrete approximation scheme for the
price of asset that is modeled by the Cox--Ingersoll--Ross process. In
order to construct these additive and multiplicative processes, we take
the Euler approximations of the CIR process itself but replace the
increments of the Wiener process with iid bounded vanishing symmetric
random variables. We introduce a ``truncated'' CIR process and use it
to prove the weak convergence of asset prices.

The paper is organized as follows. In Section~\ref{section2}, we
present a complete and ``truncated'' CIR process and establish that the
``truncated'' CIR process can be described as the unique strong
solution to the corresponding stochastic differential equation. We
establish that this ``truncated'' process does not hit zero under the
same condition as for the original nontruncated process. In Section~\ref{section3}, we present discrete approximation schemes for both these
processes and prove the weak convergence of asset prices for the
additive model. In the next section, we prove the weak convergence of
asset prices for the multiplicative model. Appendix contains additional
and technical results.\looseness=1


\section{Original and ``truncated'' Cox--Ingersoll--Ross processes and
some of their\hfill\break properties}\label{section2}
Let $\Om_{\F}= (\Om,\F,(\F_t,t\ge0),\pr )$ be a complete
filtered probability space, and $W= \{W_t,\F_t,t\ge0 \}$ be an
adapted Wiener process. Consider a Cox--Ingersoll--Ross process with
constant parameters on this space. This process is described as the
unique strong solution of the following stochastic differential equation:
\begin{equation}
\label{CIR} dX_t=(b-X_t)dt+\si\sqrt{X_t}dW_t, \quad X_0=x_0>0, \; t \ge0,
\end{equation}
where $b>0$, $\si>0$.
The integral form of the process $X$ has the following form:
\begin{equation*}
\label{explicit_CIR} X_t=x_0+\int\limits
_0^t
(b-X_s )ds+\si\int\limits
_0^t\sqrt{X_s}dW_s.
\end{equation*}
According to the paper \cite{Cox_Ing_Ross}, the condition $\si^2\le2b$
is necessary and sufficient for the process $ X $ to get positive
values and not to hit zero. Further, we will assume that this condition
is satisfied.

For the proof of functional limit theorems, we will need a modification
of the Cox--Ingersoll--Ross process with bounded coefficients. This
process is called a truncated Cox--Ingerssol--Ross process.
Let $C>0$. Consider the following stochastic differential equation with
the same coefficients $b$ and $\si$ \xch{as in}{as in Eq.}~\eqref{CIR}:
\begin{equation}
\label{procc} dX_t^C=\bigl(b-X_t^C \wedge C\bigr)dt+\si\sqrt{\bigl(X_t^C\vee0\bigr) \wedge C}dW_t, \quad X_0=x_0 >0, \; t \ge0.
\end{equation}
\begin{lem}\label{lem1}
For any \xch{$C>0$,}{$C>0$, Eq.}~\eqref{procc} has a unique strong solution.
\end{lem}
\begin{proof}
Since the coefficients $\si(x)=\si\sqrt{(x\vee0)\wedge C}$ and
$b(x)=b-(x\wedge C)$ satisfy the conditions of Theorem~\ref{watanabe2}
and also the growth condition \eqref{gr_cond}, a~global strong solution
$X_t^C$ exists uniquely for every given initial value $x_0$.
\end{proof}
\begin{zau}
Denote $\si_{-\epsilon}=\inf{ \{t: X_t^C=-\epsilon \}}$ with
$\epsilon>0$ such that $-\epsilon+b>0$. Suppose that $\pr(\si_{-\epsilon
}<\infty)>0$. Then for any $r<\si_{-\epsilon}$ such that $X_t^C<0$ for
$t\in(r,\si_{-\epsilon})$, we would have, with positive probability,
\[
dX_t^C=\bigl(b-X_t^C\wedge C
\bigr)dt>0
\]
on the interval $(r,\si_{-\epsilon})$, and hence $t\rightarrow X_t^C$
would increase in this interval. This is obviously impossible.
Therefore, $X_t^C$ is nonnegative and can be written as
\begin{equation}
\label{procc1}
dX_t^C=\bigl(b-X_t^C \wedge C\bigr)dt+\si\sqrt{X_t^C\wedge C}dW_t, \quad X_0^C=x_0>0, \ t \ge0.
\end{equation}
\end{zau}\eject
The integral form of the process $X^C$ is as follows:
\begin{equation*}
\label{truncated_explicit_CIR} X_t^C=x_0+\int\limits
_0^t
\bigl(b-X_s^C\wedge C\bigr)ds+\si\int\limits
_0^t
\sqrt {X_s^C\wedge C}dW_s.
\end{equation*}
\begin{lem}\label{lem2}
Let $2b\ge\si^2$ and $C>b\vee1$. Then the trajectories of the process
$X^C$ are positive with probability 1.
\end{lem}
\begin{proof}
In order to prove that the process $X^C$ is positive, we will use the
proof similar to that given in \cite[p. 308]{Mao} for the complete
Cox--Ingersoll--Ross process with corresponding modifications. Note
that the coefficients $g(x):=\si\sqrt{x\wedge C}$ and $f(x):=b-x\wedge
C$ \xch{of}{of Eq.}~\eqref{procc1} are continuous and $g^2(x)>0$ on $x\in(0,\infty
)$. Fix $\alpha$ and $\beta$ such that $0<\alpha<x_0<\beta$. Due to the
nonsingularity of $g $ on $[\alpha,\beta]$, there exists a unique
solution $F(x)$ of the ordinary differential equation
\[
f(x)F'(x)+\frac{1}{2}g^2(x)F''(x)=-1, \quad  \alpha<x<\beta,
\]
with boundary conditions $F(\alpha)=F(\beta)=0$, and this solution is
nonnegative, which follows from its representation through a
nonnegative Green function given in \cite[p.~343]{Karatzas}. Define the
stopping times
\[
\tau_\alpha=\inf \bigl\{t\ge0: X_t^C\le\alpha
\bigr\} \quad \text{and} \quad \tau_\beta=\inf \bigl\{t\ge0:
X_t^C\ge\beta \bigr\}.
\]
By the It\^o formula, for any $t>0$,
\begin{equation}
\label{equation1} \M F \bigl(X^C(t\wedge\tau_\alpha\wedge
\tau_\beta) \bigr)=F(x_0)-\M (t\wedge\tau_\alpha
\wedge\tau_\beta).
\end{equation}
This formula and nonnegativity of $ F $ imply that
\[
\M(t\wedge\tau_\alpha\wedge\tau_\beta)\le F(x_0)
\]
and, as $t\rightarrow\infty$,
\[
\M(\tau_\alpha\wedge\tau_\beta)\le F(x_0)<\infty.
\]
This means that $X^C$ exits from every compact subinterval of $[\alpha,
\beta]\subset(0,\infty)$ in finite time. It follows from the boundary
conditions and equality $\pr(\tau_\alpha\wedge\tau_\beta<\infty)=1$
that $\lim_{t\rightarrow\infty}\M F (X^C(t\wedge\tau_\alpha
\wedge\tau_\beta) )=0$, and then from \eqref{equation1} we have
\begin{equation*}
\M(\tau_\alpha\wedge\tau_\beta)=F(x_0).
\end{equation*}
Let us now define the function
\[
V(x)=\int\limits
_{1}^{x}\exp \biggl\{-\int\limits
_{1}^{y}
\frac
{2f(z)}{g^2(z)}dz \biggr\}dy, \quad x\in(0,\infty),
\]
which has a continuous strictly positive derivative $V'(x)$, and the
second derivative $V''(x)$ exists and satisfies
$V''(x)=-\frac{2f(x)}{g^2(x)}V'(x)$.
The It\^o formula shows that, for any $t>0$,\vadjust{\eject}
\[
V \bigl(X^C(t\wedge\tau_\alpha\wedge\tau_\beta)
\bigr) =V(x_0)+\int\limits
_{0}^{t\wedge\tau_\alpha\wedge\tau_\beta}V'
\bigl(X^C_u \bigr)g \bigl(X^C_u
\bigr)dW(u)
\]
and
\[
\M V \bigl(X^C(t\wedge\tau_\alpha\wedge\tau_\beta)
\bigr) =V(x_0).
\]
Taking the limit as $t\rightarrow\infty$, we get\begingroup\abovedisplayskip=12pt\belowdisplayskip=12pt
\[
V(x_0)=\M V\bigl(X^C(\tau_\alpha\wedge
\tau_\beta)\bigr)=V(\alpha)\pr(\tau_\alpha <\tau_\beta)+V(
\beta)\pr(\tau_\beta<\tau_\alpha),
\]
and hence
\begin{equation}
\label{equation2} \pr(\tau_\alpha<\tau_\beta)=\frac{V(\beta)-V(x_0)}{V(\beta)-V(\alpha)}
\quad \text{and} \quad \pr(\tau_\beta<\tau_\alpha)=\frac{V(x_0)-V(\alpha
)}{V(\beta)-V(\alpha)}.
\end{equation}
Consider the integral
\begin{equation*}
\begin{split} \label{compute} V(x)&=\int\limits
_{1}^{x}\exp
\biggl\{-\int\limits
_{1}^{y}\frac
{2(b-z\wedge C)}{\si^2(z\wedge C)}dz \biggr\}dy. \end{split}
\end{equation*}
First, consider the case $x<1 $. Then
\begin{equation*}
\begin{split} \label{compute1} V(x)&=\int\limits
_{1}^{x}\exp
\biggl\{-\int\limits
_{1}^{y}\frac{2(b-z)}{
\si^2z}dz \biggr\}dy=\int\limits
_{1}^{x}y^{-\frac{2b}{\si^2}}
\exp \biggl\{ \frac{2(y-1)}{\si^2} \biggr\}dy, \end{split} %
\end{equation*}
and if $\si^2\le2b$, then
\[
\lim\limits
_{x\downarrow0}V(x)=-\infty.
\]
Now let $x$ increase and tend to infinity. Denote
$C_1=\int_{1}^{C}\exp \{\frac{2(y-1)}{\si^2} \}
y^{-\frac{2b}{\si^2}}dy$. Then, for $x>C$,\endgroup
%
\begin{align*}
\label{compute2}
V(x)&=\int\limits_{1}^{C}\exp \biggl\{-\int\limits_{1}^{y}\frac{2(b-z)}{\si^2z}dz \biggr\}dy\\
&\quad +\int\limits_{C}^{x}\exp \biggl\{-\int\limits_{1}^{C}\frac{2(b-z)}{\si^2z}dz-\int\limits_{C}^{y}\frac{2(b-C)}{\si^2C}dz \biggr\}dy\\
&=\int\limits_{1}^{C}\exp \biggl\{\frac{2(y-1)}{\si^2} \biggr\}y^{-\frac{2b}{\si^2}}dy +C^{-\frac{2b}{\si^2}}\exp{ \biggl\{\frac{2(C-1)}{\si^2} \biggr\}}\\
&\quad \times\int\limits_{C}^{x}\exp \biggl\{-\frac{2(b-C)}{\si^2C}(y-C)\biggr\} dy\\
&=C_1+C^{-\frac{2b}{\si^2}+1}\frac{\si^2}{2(C-b)} \exp \biggl\{\frac{2(C-1)}{\si^2} \biggr\}\\
&\quad \times \biggl(\exp \biggl\{\frac{2(C-b)}{\si^2C}(x-C) \biggr\}-1 \biggr),
\end{align*}
and thus $\lim_{x\uparrow\infty}V(x)=\infty$.
Define
\[
\tau_0=\lim\limits_{\alpha\downarrow0}\tau_\alpha \quad \text{and} \quad \tau_\infty=\lim\limits_{\beta\uparrow\infty}\tau_\beta
\]
and put $\tau=\tau_0\wedge\tau_\infty$. From \eqref{equation2} we get
\[
\pr \Bigl(\inf\limits
_{0\le t <\tau}X_t^C\le\alpha \Bigr)\ge\pr(\tau
_\alpha<\tau_\beta)=\frac{1-V(x_0)/V(\beta)}{1-V(\alpha)/V(\beta)},
\]
and, as $\beta\uparrow\infty$, we get that, for any $\alpha>0$,
$\pr (\inf_{0\le t <\tau}X_t^C\le\alpha )=1$,
whence, finally,
$\pr (\inf_{0\le t <\tau}X_t^C=0 )=1$.
Similarly,
$\pr (\sup_{0\le t <\tau}X_t^C=\infty )=1$.
Assume now that $\pr(\tau<\infty)>0$. Then
\[
\pr \Bigl(\lim\limits_{t\rightarrow\tau}X_t^C \ \text{exists and equals}\ 0 \ \text{or}\ \infty \Bigr)>0.
\]
So the events $ \{\inf_{0\le t <\tau}X_t^C=0 \}$ and
$ \{\sup_{0\le t <\tau}X_t^C=\infty \}$ cannot both
have probability 1. This contradiction shows that $\pr(\tau<\infty)=0$, whence
\[
\pr(\tau=\infty)=\pr \Bigl(\inf\limits
_{0\le t <\tau}X_t^C=0 \Bigr)=
\pr \Bigl(\sup\limits
_{0\le t <\tau}X_t^C=\infty \Bigr)=1
\]
if $2b\ge\si^2$.
\end{proof}
Now, let $T>0$ be fixed.
\begin{lem}
\[
\pr\bigl\{\exists\, t\in[0,T]: X_t\neq X_t^C
\bigr\}\rightarrow0
\]
as $C\rightarrow\infty$.
\end{lem}
\begin{proof}
Obviously, it suffices to show
that
\[
\pr \Bigl\{\sup\limits
_{t\in[0,T]}|X_t|\ge C \Bigr\}\rightarrow0 \quad \text{as} \; C\rightarrow\infty.
\]
It is well known (see, e.g., \cite{Zhu}) that $\frac{4}{\si
^2(1-e^{-t})}X_t$ follows a noncentral $\chi^2$ distribution with (in
general) noninteger degree of freedom $\frac{4b}{\si^2}$ and
noncentrality parameter $\frac{4}{\si^2(1-e^{-t})}x_0 e^{-t}$. The
first and second moments for any $t\ge0$ are given by
\[
\M X_t=x_0e^{-t}+b\bigl(1-e^{-t}
\bigr),
\]
\[
\M(X_t)^2=x_0 \bigl(2b+\si^2
\bigr)e^{-t}+ \bigl(x_0^2-x_0\si
^2-2x_0b \bigr)e^{-2t}+ \biggl(
\frac{b\si^2}{2}+b^2 \biggr) \bigl(1-e^{-t}
\bigr)^2.
\]\eject
\noindent Therefore, there exists a constant $B>0$ such that $\M X_t^2\le B$,
whence $\M X_t\le B^{1/2}$, $0\le t\le T$.

Using the Doob inequality, we estimate
\begin{align*}
&\pr \Bigl\{\sup\limits_{t\in[0,T]}|X_t|\ge C \Bigr\}\le\frac{1}{C^2}\M \sup\limits_{t\in[0,T]}X_t^2\\
&\quad =\frac{1}{C^2}\M\sup\limits_{t\in[0,T]} \biggl\{ \biggl(X_0+\int\limits_0^t (b-X_s )ds+\si\int\limits_0^t\sqrt{X_s}dW_s \biggr)^2 \biggr\}\\
&\quad \le\frac{3}{C^2} \biggl\{X_0^2+T\M \biggl(\int\limits_0^T\llvert b-X_s\rrvert ds \biggr)^2+\si^2\M\sup\limits_{t\in[0,T]} \biggl(\int\limits_0^t\sqrt {X_s}dW_s \biggr)^2 \biggr\}\\
&\quad \le\frac{3}{C^2} \biggl\{X_0^2+T\M\int\limits_0^T (b-X_s )^2ds+4\si^2\M\int\limits_0^TX_sds \biggr\}\le\frac{B_1}{C^2}
\end{align*}
for some constant $B_1>0$.
The lemma is proved.
\end{proof}

\section{Discrete approximation schemes for complete and ``truncated''\hfill\break Cox--Ingersoll--Ross processes}\label{section3}
Consider the following discrete approximation scheme for the process
$X$. Assume that we have a sequence of the probability spaces $(\Om
^{(n)},\F^{(n)},\pr^{(n)})$, $n\ge1$. Let $\{q_k^{(n)}, n\ge1$, $0\le
k \le n\}$ be the sequence of symmetric iid random variables defined on
the corresponding probability space and taking values~ $\pm\sqrt{\frac
{T}{n}}$, that is, $\pr^n (q_k^{(n)}=\pm\sqrt{\frac{T}{n}}
)=\frac{1}{2}$. Let further $n>T$.
We construct discrete approximation schemes for the stochastic
processes $X$ and $X^C$ as follows. Consider the following
approximation for the complete process:
\begin{align}
\label{rec_sch_CIR_E}
&X_0^{(n)}=x_0>0, \qquad X_{k}^{(n)} = X_{k-1}^{(n)}+\frac{(b-X_{k-1}^{(n)})T}{n}+\si q_k^{(n)}\sqrt{X^{(n)}_{k-1}},\nonumber\\
&Q_k^{(n)}:=X_k^{(n)}-X_{k-1}^{(n)} = \frac{(b-X_{k-1}^{(n)})T}{n}+\si q_k^{(n)}\sqrt{X^{(n)}_{k-1}}, \quad 1\le k \le n,
\end{align}
and the corresponding approximations for $X^C$ given by
\begin{align}
X_0^{(n,C)}&=x_{0}>0,\nonumber\\
X_k^{(n,C)}&=X_{k-1}^{(n,C)}+\frac{(b-(X_{k-1}^{(n,C)}\wedge C))T}{n}+\si q_k^{(n)}\sqrt{X^{(n,C)}_{k-1}\wedge C},\nonumber\\
Q_k^{(n,C)}:&=X_k^{(n,C)}-X_{k-1}^{(n,C)}\nonumber\\
&=\frac{(b-(X_{k-1}^{(n,C)}\wedge C))T}{n}+\si q_k^{(n)}\sqrt{X^{(n,C)}_{k-1}\wedge C}, 1\le k \le n.\label{rec_sch_E}
\end{align}
The following lemma confirms the correctness of the construction of
these approximations.
\begin{lem}
Let, $n>2T$.
\begin{itemize}
\item[$1)$] If $2b\geq\sigma^2$, then all values given \xch{by}{by Eqs.}~\eqref
{rec_sch_CIR_E} and \eqref{rec_sch_E} are positive.
\item[$2)$] We have
\begin{equation}
\label{eq11} \pr \bigl\{\exists k, 0\le k\le n: \; X_k^{(n)}
\neq X_k^{(n,C)} \bigr\} \rightarrow0
\end{equation}
as $C\rightarrow\infty$.
\end{itemize}
\end{lem}
\begin{proof}
$1)$ We apply the method of mathematical induction. When $k=1$,
\[
X_1^{(n)}=x_{0}+\frac{(b-x_{0})T}{n}+\si
q_1^{(n)}\sqrt{x_{0}}.
\]
Let us show that
\begin{equation}
\label{rec_sch_E11}x_{0}+\frac{(b-x_{0})T}{n}+\si q_1^{(n)}
\sqrt{x_{0}}>0.
\end{equation}
We denote $\alpha:=\sqrt{x_0}$ and reduce \eqref{rec_sch_E11} to the
quadratic inequality
\[
\alpha^2 \biggl(1-\frac{T}{n} \biggr)\pm\si\sqrt{
\frac{T}{n}}\alpha+\frac
{bT}{n}>0,
\]
which obviously holds because the discriminant $D=\frac{\si^2T}{n}-\frac
{4bT}{n} (1-\frac{T}{n} )<0$ when $\si^2\le2b$ and $n>2T $.
So, $X_1^{(n)}>0$. Assume now that $X_{k}^{(n)}>0$. It can be shown by
applying the same transformation that when $\si^2\le2b$ and $n>2T $,
the values $X_{k+1}^{(n)}>0$.

It can be proved similarly that the values given by \eqref{rec_sch_E}
are positive.

2) $X_k^{(n)}$ can be represented as
\begin{align}
X_k^{(n)}&=x_0+\sum\limits_{i=1}^k\bigl(b-X_{i-1}^{(n)}\bigr)\frac{T}{n}+\si\sum\limits_{i=1}^k q_i^{(n)}\sqrt{X_{i-1}^{(n)}}\nonumber\\
&=X_{k-1}^{(n)}+\frac{(b-X_{k-1}^{(n)})T}{n}+\si q_k^{(n)}\sqrt{X^{(n)}_{k-1}}. \label{eq_X_k^n}
\end{align}
Compute
\begin{align}
\M\bigl(X_{i}^{(n)} \bigr)^2 & =\M \biggl(X_{i-1}^{(n)} \biggl(1- \frac{T}{n} \biggr)+\frac {bT}{n}+\si\sqrt{X_{i-1}^{(n)}}q_i^{(n)} \biggr)^2\nonumber\\
& =\M \biggl(X_{i-1}^{(n)} \biggl(1-\frac{T}{n} \biggr)+ \frac{bT}{n} \biggr)^2+\frac{\si^2T}{n}\M X_{i-1}^{(n)} \nonumber\\
& = \biggl(\frac{bT}{n} \biggr)^2+ \biggl[\frac{\si^2T}{n}+ \frac{2bT}{n} \biggl(1-\frac{T}{n} \biggr) \biggr]\M X_{i-1}^{(n)} \nonumber\\
& \quad + \biggl(1-\frac{T}{n} \biggr)^2\M \bigl(X_{i-1}^{(n)} \bigr)^2.\label{EXkn2}
\end{align}
Assume that $\M (X_{j}^{(n)} )^2\le\beta^2$, $1\le j \le i-1$,
for some $\beta>0$. Then $\M X_{j}^{(n)}\le\beta$, $1\le j \le i-1$. We
get the quadratic inequality of the form
\[
\biggl(1-\frac{T}{n} \biggr)^2\beta^2+ \biggl[
\frac{\si^2T}{n}+\frac
{2bT}{n} \biggl(1-\frac{T}{n} \biggr) \biggr]
\beta+ \biggl(\frac{bT}{n} \biggr)^2<\beta^2
\]
or, equivalently,
\[
\biggl( \biggl(1-\frac{T}{n} \biggr)^2-1 \biggr)
\beta^2+ \biggl[\frac{\si
^2T}{n}+\frac{2bT}{n} \biggl(1-
\frac{T}{n} \biggr) \biggr]\beta+ \biggl(\frac
{bT}{n}
\biggr)^2<0,
\]
which obviously holds when $\beta>\frac{\si^2+2b+\sqrt{\si^4+4b\si
^2+8b^2}}{\frac{3}{2}}$. So, for all $1\le i \le n$,
$\M X_{i}^{(n)}\le\frac{\si^2+2b+\sqrt{\si^4+4b\si^2+8b^2}}{\frac
{3}{2}} \vee x_0=:\gamma$.

Using the Burkholder inequality, we estimate
\begin{align*}
0\le\M\sup\limits_{0\le k \le n} \bigl(X_k^{(n)} \bigr)^2&\le2 (x_0+bT )^2+ 2\si^2\M\sup\limits_{0\le k \le n}\Biggl(\sum\limits_{i=1}^{n} q_{i}^{(n)}\sqrt{X_{i-1}^{(n)}} \Biggr)^2\\
&\le2 (x_0+bT )^2+ 8\si^2\M \Biggl(\sum\limits_{i=1}^{n} q_{i}^{(n)}\sqrt{X_{i-1}^{(n)}} \Biggr)^2\\
&\le2 (x_0+bT )^2+8\sigma^2\gamma T.
\end{align*}
Therefore,
\begin{align*}
\pr \bigl\{\exists k, 0\le k \le n: X_k^{(n)}\neq X_k^{(n,C)} \bigr\}&=\pr \Bigl\{\sup_{0\le k \le n}X_k^{(n)}\geq C\Bigr\}\leq C^{-2}\M\sup_{0\le k \le n}\bigl(X^{(n)}_k\bigr)^2\\
&\leq 2C^{-2}(x_0+bT)^2+8\sigma^2C^{-2}\gamma T,
\end{align*}
whence the proof follows.
\end{proof}
Consider the sequences of step processes corresponding to these schemes:
\begin{equation*}
\label{eq_disproc} X^{(n)}_t=X^{(n)}_{k} \;
\text{for } \frac{kT}{n} \le t < \frac{(k+1)T}{n}
\end{equation*}
and
\begin{equation*}
\label{eq_disproc1} X^{ (n,C) }_t=X^{(n,C)}_{k} \;
\text{for } \frac{kT}{n} \le t < \frac
{(k+1)T}{n}.
\end{equation*}
Thus, the trajectories of the processes $X^{(n)}$ and $X^{(n,C)}$ have
jumps at the points
$kT/n\;,k=0,\ldots, n$, and are constant on the interior intervals.
Consider the filtrations $\F_k^{n}=\sigma(X^{(n)}_t,\, t\le\frac
{kT}{n} )$. The processes $X^{(n,C)}$ are adapted with respect to them.
Therefore, we can consider the same filtrations for all discrete
approximation schemes. So, we can identify $\F^n_t$ with $\F_k^{n}$ for
$\frac{kT}{n} \le t < \frac{(k+1)T}{n}$.
\begin{zau}\label{zau3.1}
Now we can rewrite relation \eqref{eq11} as follows:
\[
\pr \bigl\{\exists t, t\in[0,T]: X_t^{(n)}\neq
X_t^{(n,C)} \bigr\} \rightarrow0
\]
as $C\rightarrow\infty$.
\end{zau}
Denote by $\mathbb{Q}$ and $\mathbb{Q}^n, n\ge1$, the measures
corresponding to the processes
$X$ and $X^{(n)}, n\ge1$, respectively, and by $\mathbb{Q}^C$ and
$\mathbb{Q}^{n,C}, n\ge1$, the measures corresponding to the processes
$X^C$ and $X^{(n,C)}, n\ge1$, respectively. Denote by $\stackrel
{W}{\longrightarrow}$ the weak convergence of measures corresponding to
stochastic processes. We apply Theorem~3.2 from \cite{Mishura3} to
prove the weak convergence of measures $\mathbb{Q}^{n,C}$ to the
measure $\mathbb{Q}^C$. This theorem can be formulated as follows.
\begin{thm}
\label{conditions_E}
\begingroup
\abovedisplayskip=6pt
\belowdisplayskip=6pt
Assume that the following conditions are satisfied:
\begin{itemize}
\item[\rm(i)] For any $\epsilon>0$,
\[
\lim\limits
_{n}\pr \Bigl(\sup\limits
_{1\le k \le n}\big|Q_k^{(n,C)}\big|
\ge \epsilon \Bigr)=0;
\]
\item[\rm(ii)] For any $\epsilon>0$ and $a\in (0,1] $,
\begin{align*}
&\lim\limits_{n}\pr^n \biggl(\sup\limits_{t\in\T}\biggl\llvert\sum\limits_{1\le k\le [\frac{nt}{T} ]}\M \bigl(Q_k^{(n,C)}\I_{|Q_k^{(n,C)}|\le a} \big\rrvert \F_{k-1}^n \bigr)\\
&\qquad -\int\limits_{0}^{t}\bigl(b-X_s^{(n,C)}\wedge C\bigr)ds\biggr\rrvert \ge \epsilon \biggr)=0;
\end{align*}
\item[\rm(iii)] For any $\epsilon>0$ and $a\in (0,1] $,
\begin{align*}
&\lim\limits_{n}\pr^n \biggl(\sup\limits_{t\in\T}\biggl\llvert\sum\limits_{1\le k\le [ \frac{nt}{T} ]} \Bigl(\M \bigl( \bigl(Q_k^{(n,C)}\bigr)^2\I_{|Q_k^{(n,C)}|\le a} \big\rrvert \F_{k-1}^n\bigr) \\
&\qquad - \bigl(\M \bigl(Q_k^{(n,C)}\I_{|Q_k^{(n,C)}|\le a} \big\rrvert\F_{k-1}^n \bigr) \bigr)^2 \Bigr)-\si^2\int\limits_{0}^{t}\bigl(X_s^{(n,C)}\wedge C\bigr) ds\biggr\rrvert \ge\epsilon \biggr)=0;
\end{align*}
\end{itemize}
\endgroup
Then
$\mathbb{Q}^{n,C}\stackrel{W}{\longrightarrow}\mathbb{Q}^C$.
\end{thm}
Using Theorem~\ref{conditions_E}, we prove the following result.
\begin{thm}
\label{weak_conv_E}
$\mathbb{Q}^{n,C}\stackrel{W}{\longrightarrow}\mathbb{Q}^C$.
\end{thm}
\begin{proof}
According to Theorem~\ref{conditions_E}, we need to check conditions
(i)--(iii). Relation \eqref{rec_sch_E} implies that
$\sup_{0\le k \le n}{|Q_k^{(n,C)}|}\le\frac{b+CT}{n}+\si\sqrt
{\frac{TC}{n}}$. Hence, there exists a~constant $C_2>0$ such that $\sup_{0\le k\le n}{|Q_k^{(n,C)}|}\le\frac{C_2}{\sqrt{n}}$. This
means that condition~(i) is satisfied.

Furthermore, in order to establish (ii), we consider any fixed $a>0$
and $n\geq1$ such that $\frac{C_2}{\sqrt{n}}\le a$, that is,
$n\ge (\frac{C_2}{a} )^{2}$. For such $n$,
\begin{align}
\M \bigl(Q_k^{(n,C)}\I_{|Q_k^{(n,C)}|\le a} \big\rrvert \F_{k-1}^n \bigr)&=\M\bigl(Q_k^{(n,C)} \big\rrvert \F_{k-1}^n\bigr)\nonumber\\
&=\frac{(b-(X_{k-1}^{(n,C)}\wedge C))T}{n}+\si\M q_k^{(n)}\sqrt{X^{(n,C)}_{k-1}\wedge C}\nonumber\\
&=\frac{(b-(X_{k-1}^{(n,C)}\wedge C))T}{n}.\label{exp_Q}
\end{align}
For any $\epsilon>0$, we have\vspace*{-9pt}
\begin{align*}
&\lim\limits_{n}\pr^n \biggl(\sup\limits_{t\in\T}\biggl\llvert\sum\limits_{1\le k\le [\frac{nt}{T} ]}\M \bigl(Q_k^{(n,C)}\I_{|Q_k^{(n,C)}|\le a} \bigr\rrvert \F_{k-1}^n \bigr)-\int\limits_{0}^{t}\bigl(b-\bigl(X_s^{(n,C)}\wedge C\bigr)\bigr)ds\biggr\rrvert \ge\epsilon \biggr)\\
&\quad =\lim\limits_{n}\pr^n \biggl(\sup\limits_{t\in\T}\biggl\llvert\sum\limits_{1\le k\le [\frac{nt}{T} ]}\frac{(b-(X_{k-1}^{(n,C)}\wedge C))T}{n}-\sum\limits_{0\le k\le [\frac{nt}{T}]-1}\bigl(b-\bigl(X_{k}^{(n,C)}\wedge C\bigr)\bigr)\frac{T}{n}\\
&\qquad -\bigl(b-\bigl(X_{ [\frac{nt}{T} ]}^{(n,C)}\wedge C\bigr)\bigr) \biggl(t-\frac{ [\frac{nt}{T} ]T}{n} \biggr)\biggr\rrvert \ge\epsilon \biggr)\\
&\quad =\lim\limits_{n}\pr^n \biggl(\sup\limits_{t\in\T}\biggl\llvert\bigl(b-\bigl(X_{ [\frac{nt}{T} ]}^{(n,C)}\wedge C\bigr)\bigr) \biggl(t-\frac{ [\frac{nt}{T}]T}{n} \biggr)\biggr\rrvert \ge\epsilon \biggr)=0,
\end{align*}
and hence condition (ii) is satisfied. Now let us check condition
(iii). We have
\begin{align*}
&\M \bigl( \bigl(Q_k^{(n,C)} \bigr)^2\I_{|Q_k^{(n,C)}|\le a} \rrvert \F _{k-1}^n \bigr)=\M \bigl(\bigl(Q_k^{(n,C)} \bigr)^2 \rrvert \F_{k-1}^n \bigr)\\
&\quad = \biggl(\frac{(b-(X_{k-1}^{(n,C)}\wedge C))T}{n} \biggr)^2+2\frac{(b-(X_{k-1}^{(n,C)}\wedge C))T}{n}\si\M q_k^{(n)}\sqrt {X^{(n,C)}_{k-1}\wedge C}\\
&\qquad +\si^2 \M\bigl(q_k^{(n)}\bigr)^2\bigl(X^{(n,C)}_{k-1}\wedge C \bigr)\\
&\quad = \biggl(\frac{(b-(X_{k-1}^{(n,C)}\wedge C))T}{n}\biggr)^2{+}\,\si^2 \frac{T}{n} \bigl(X^{(n,C)}_{k-1}\wedge C \bigr).
\end{align*}
Therefore, for any $\epsilon>0$,
\begin{align*}
&\lim\limits_{n}\pr^n \biggl(\sup\limits_{t\in\T}\biggl\llvert\sum\limits_{1\le k\le [ \frac{nt}{T} ]} \Bigl(\M \bigl( \bigl(Q_k^{(n,C)}\bigr)^2\I_{|Q_k^{(n,C)}|\le a} \bigr\rrvert \F_{k-1}^n\bigr) \\
&\qquad - \bigl(\M \bigl(Q_k^{(n,C)}\I_{|Q_k^{(n,C)}|\le a} \big\rrvert \F_{k-1}^n \bigr) \bigr)^2 \Bigr)-\si^2\int\limits_{0}^{t}\bigl(X_s^{(n,C)}\wedge C\bigr) ds\biggr\rrvert \ge\epsilon \biggr)\\
&\quad =\lim\limits_{n}\pr^n \biggl(\sup\limits_{t\in\T}\biggl\llvert\sum\limits_{1\le k\le [ \frac{nt}{T} ]} \biggl( \biggl(\frac{(b-(X_{k-1}^{(n,C)}\wedge C) )T}{n}\biggr)^2+\si^2\frac{T}{n} \bigl(X_{k-1}^{(n,C)}\wedge C \bigr)\\
&\qquad - \biggl(\frac{ (b-(X_{k-1}^{(n,C)}\wedge C))T}{n} \biggr)^2 \biggr)-\sum\limits_{0\le k\le [ \frac{nt}{T}]-1} \biggl(\si^2 \frac{T}{n}\bigl(X_{k}^{(n,C)}\wedge C \bigr) \biggr)\\
&\qquad -\si^2 \bigl(X_{ [\frac{nt}{T} ]}^{(n,C)}\wedge C \bigr)\biggl(t-\frac{ [\frac{nt}{T} ]T}{n} \biggr)\biggr\rrvert \ge \epsilon \biggr)\\
&\quad =\lim\limits_{n}\pr^{n} \biggl(\sup\limits_{t\in\T} \biggl(\si^2 \bigl(X_{ [\frac{nt}{T} ]}^{(n,C)}\wedge C \bigr) \biggl(t-\frac{[\frac{nt}{T} ]T}{n} \biggr) \biggr)\ge\epsilon \biggr)=0.
\end{align*}
The theorem is proved.
\end{proof}
\begin{thm}\label{weak_con}
$\mathbb{Q}^n\stackrel{W}{\longrightarrow}\mathbb{Q}$, $n\rightarrow
\infty$.
\end{thm}
\begin{proof}
According to Theorem~\ref{bill} and Theorem~\ref{weak_conv_E}, it
suffices to prove that
\[
\lim\limits
_{C\rightarrow\infty}\overline{\lim\limits
_{n\rightarrow
\infty}}\pr \Bigl\{\sup\limits
_{0\le t\le T}
\bigl\llvert X_t^{(n)}- X_t^{(n,C)}\bigr
\rrvert \ge\epsilon \Bigr\}=0.
\]
However, due to Remark~\ref{zau3.1},
\begin{align*}
&\lim\limits_{C\rightarrow\infty}\overline{\lim\limits_{n\rightarrow\infty}}\pr \Bigl\{\sup\limits_{0\le t\le T}\bigl\llvert X_t^{(n)}- X_t^{(n,C)}\bigr\rrvert \ge\epsilon \Bigr\}\\
&\quad \le\lim\limits_{C\rightarrow\infty}\overline{\lim\limits_{n\rightarrow\infty}}\pr \bigl\{\exists t, t\in[0,T]: X_t^{(n)}\neq X_t^{(n,C)} \bigr\}=0.\qedhere
\end{align*}
\end{proof}

\section{Multiplicative scheme for Cox--Ingersoll--Ross process}\label{section4}
In this section, we construct a multiplicative discrete approximation
scheme for the process $e^{X_t}$, $t\in[0,T]$, where $X_t$ is the CIR
process given \xch{by}{by Eq.}~\eqref{explicit_CIR}. We construct the following
multiplicative process based on the discrete approximation scheme \mbox{\eqref
{rec_sch_CIR_E}--\eqref{rec_sch_E}}. We introduce limit and prelimit
processes as follows:
\begin{flalign*}
&S_t^{n,C}=\exp \{x_0 \}\prod\limits_{1\le k \le [\frac{tn}{T} ]} \bigl(1+Q_k^{(n,C)} \bigr), \quad  t\in\mathbb{T},&\\
&S_t^C=\exp{ \biggl\{X_t^C-\frac{\si^2}{2}\int\limits_0^t\bigl(X_t^C\wedge C\bigr)dt \biggr\}}, \quad  t\in\mathbb{T},&\\
&S_t^{n}=\exp \{x_0 \}\prod\limits_{1\le k \le [\frac{tn}{T} ]} \bigl(1+Q_k^{(n)} \bigr), \quad t\in \mathbb{T},&\\
&S_t=\exp{ \biggl\{X_t-\frac{\si^2}{2}\int\limits_0^t X_t dt \biggr\}}, \quad t\in\mathbb{T},&\\
&\widetilde{S}_t^{n}=\exp \{x_0 \}\prod\limits_{1\le k \le[\frac{tn}{T} ]} \biggl[ \bigl(1+Q_k^{(n,C)} \bigr)\exp\biggl\{\frac{\si^2}{2n}X_k^{(n)} \biggr\} \biggr], \quad t \in\mathbb{T},&\\
\mbox{and}\qquad \qquad &&\\
&\widetilde{S}_t=\exp \{X_t \}, \quad t\in\mathbb{T}.&
\end{flalign*}
Denote by $\G^{C}$, $\G^{n,C}$, $\G$, $\G^{n}$, $\widetilde{\G}$, and
$\widetilde{\G}^{n}$, $n\ge1$, the measures corresponding to the
processes $S_t^{C}$, $S_t^{n,C}$, $S_t$, $S_t^{n}$, $\widetilde{S}_t$,
and $\widetilde{S}_t^{n}$, $n\ge1$, respectively.

We apply Theorem~3.3 from \cite{Mishura3} to prove the weak convergence
of measures. This theorem can be formulated as follows.
\begin{thm}
\label{conditions}
Let the following conditions hold:
\begin{itemize}
\item[\rm(i)] $\sup_{1\le k \le n}|Q_k^{(n,C)}|\stackrel{\pr
}{\longrightarrow}0,\; n\rightarrow\infty$;
\item[\rm(ii)] For any $a\in (0,1] $,
\[
\lim\limits
_{D\rightarrow\infty}\overline{\lim\limits
_{n\rightarrow
\infty}}\pr^n \biggl(\sum
\limits
_{1\le k\le n}\M \bigl( \bigl(Q_k^{(n,C)}
\bigr)^2\I_{|Q_k^{(n,C)}|\le a} \big\rrvert \F _{k-1}^n
\bigr)\ge D \biggr)=0;
\]
\item[\rm(iii)] For any $a\in (0,1] $,
\[
\lim\limits
_{D\rightarrow\infty}\overline{\lim\limits
_{n\rightarrow
\infty}}\pr^n \biggl(\sum
\limits
_{1\le k\le n}\bigl\llvert \M \bigl(Q_k^{(n,C)}\I
_{|Q_k^{(n,C)}|\le a} \bigr\rrvert \F_{k-1}^n \bigr)\big\rrvert \ge D
\biggr)=0;
\]
\item[\rm(iv)] For any $\epsilon>0$ and $a\in (0,1] $,
\begin{align*}
&\lim\limits_{n}\pr^n \biggl(\sup\limits_{t\in\T}\biggl\llvert\sum\limits_{1\le k\le [\frac{nt}{T} ]}\M \bigl(Q_k^{(n,C)}\I_{|Q_k^{(n,C)}|\le a} \bigr\rrvert \F_{k-1}^n \bigr)\\
&\qquad -\int\limits_{0}^{t}\bigl(b-X_s^{(n,C)}\wedge C\bigr)ds\biggr\rrvert \ge \epsilon \biggr)=0;
\end{align*}
\item[\rm(v)] For any $\epsilon>0$ and $a\in (0,1] $,
\begin{align*}
&\lim\limits_{n}\pr^n \biggl(\sup\limits_{t\in\T}\biggl\llvert\sum\limits_{1\le k\le [ \frac{nt}{T} ]}\M \bigl( \bigl(Q_k^{(n,C)}\bigr)^2\I _{|Q_k^{(n,C)}|\le a} \bigr\rrvert \F_{k-1}^n\bigr)\\
&\qquad -\si^2\int\limits_{0}^{t}\bigl(X_s^{(n,C)}\wedge C\bigr) ds\biggr\rrvert \ge \epsilon \biggr)=0.
\end{align*}
\end{itemize}
Then
\[
\G^{n,C}\stackrel{W} {\longrightarrow}\G^C.
\]
\end{thm}
We prove the following result using Theorem~\ref{conditions}.
\begin{thm}
\label{weak_conv}
$\G^{n,C}\stackrel{W}{\longrightarrow}\G^C$.
\end{thm}
\begin{proof}
According to Theorem~\ref{conditions}, we need to check
conditions (i)--(v). It was established in the proof of Theorem~\ref
{weak_conv_E} that conditions (i) and (iv) are satisfied. Let us show
that condition (ii) holds. It was also established in the proof of
Theorem~\ref{weak_conv_E} that $\sup_{0\le k\le
n}{|Q_k^{(n,C)}|}\le\frac{C_2}{\sqrt{n}}$. So, for all $a\in
(0,1] $, starting from some number $n$, we have
\begin{equation*}
\begin{split}
&\sum\limits_{1\le k\le n}\M \bigl(\bigl(Q_k^{(n,C)} \bigr)^2\I _{|Q_k^{(n,C)}|\le a}\big\rrvert \F_{k-1}^n \bigr)\\
&\quad =\sum\limits_{1\le k\le n}\M \bigl( \bigl(Q_k^{(n,C)}\bigr)^2 \big\rrvert \F_{k-1}^n \bigr)\le\sum\limits_{1\le k\le n}\frac{C_2^2}{n} \le C_2^2,
\end{split} %
\end{equation*}
whence condition (ii) holds. Now, \eqref{exp_Q} implies that, for all
$a\in (0,1] $, starting from some number $n$, we have
\[
\bigl\llvert \M \bigl(Q_k^{(n,C)}\I_{|Q_k^{(n,C)}|\le a} \bigr
\rrvert \F _{k-1}^n \bigr)\big\rrvert =\bigl\llvert \M
\bigl(Q_k^{(n,C)} \bigr\rrvert \F _{k-1}^n
\bigr)\big\rrvert \le\frac{C_3}{n},
\]
whence condition (iii) holds.

Let us check condition (v). For any $\epsilon>0$ and $a\in
.(0,1] $, we have
\begin{align*}
&\lim\limits_{n}\pr^n \biggl(\sup\limits_{t\in\T}\biggl\llvert\sum\limits_{1\le k\le [ \frac{nt}{T} ]}\M \bigl( \bigl(Q_k^{(n,C)}\bigr)^2\I _{|Q_k^{(n,C)}|\le a} \bigr\rrvert \F_{k-1}^n\bigr)-\si^2\int\limits_{0}^{t}\bigl(X_s^{(n,C)}\wedge C\bigr) ds\biggr\rrvert \ge\epsilon \biggr)\\
&\quad =\lim\limits_{n}\pr^n \biggl(\sup\limits_{t\in\T}\biggl\llvert\sum\limits_{1\le k\le [ \frac{nt}{T} ]} \biggl( \biggl(\frac{(b-(X_{k-1}^{(n,C)}\wedge C) )T}{n}\biggr)^2+\si^2\frac{T}{n} \bigl(X_{k-1}^{(n,C)}\wedge C \bigr) \biggr)\\
&\qquad - \sum\limits_{0\le k\le [ \frac{nt}{T} ]-1} \biggl(\si^2 \frac{T}{n}\bigl(X_{k}^{(n,C)}\wedge C \bigr) \biggr)-\si^2\bigl(X_{ [\frac{nt}{T} ]}^{(n,C)}\wedge C \bigr) \biggl(t-\frac{[\frac{nt}{T} ]T}{n}\biggr)\biggr\rrvert \ge\epsilon \biggr)\\
&\quad \le\lim\limits_{n}\pr^{n} \biggl(\sup\limits_{t\in\T} \biggl(\frac{(|b|+C)^2Tt}{n}+\si^2 \bigl(X_{ [\frac{nt}{T} ]}^{(n,C)}\wedge C\bigr) \biggl(t-\frac{ [\frac{nt}{T} ]T}{n} \biggr) \biggr)\ge \epsilon \biggr)=0.
\end{align*}
The theorem is proved.
\end{proof}
\begin{thm}
$\G^n\stackrel{W}{\longrightarrow}\G, \ n\rightarrow\infty$.
\end{thm}
\begin{proof}
The proof immediately follows from Theorem~\ref{bill}, Theorem~\ref{weak_conv}, and Remark~\ref{zau3.1}. Indeed,
\begin{align*}
&\lim\limits_{C\rightarrow\infty}\overline{\lim\limits_{n\rightarrow\infty}}\pr \Bigl\{\sup\limits_{0\le t\le T}\bigl\llvert X_t^{(n)}- X_t^{(n,C)}\bigr\rrvert \ge\epsilon \Bigr\}\\
&\quad \le\lim\limits_{C\rightarrow\infty}\overline{\lim\limits_{n\rightarrow\infty}}\pr \bigl\{\exists t, t\in[0,T]: X_t^{(n)}\neq X_t^{(n,C)} \bigr\}=0.\qedhere
\end{align*}
\end{proof}
\begin{zau}
The weak convergence
\[
\widetilde{\G}^n\stackrel{W} {\longrightarrow}\widetilde{\G}, \quad n \rightarrow\infty,
\]
can be proved in a similar way.
\end{zau}

\appendix
\section{Additional results}
We state here Theorem~4.2 from \cite{Billingsley}:
\begin{thm}\label{bill}
Suppose that we have sets of processes $ \{X^{(n,C)}, n\ge1,
C>0 \}$, $ \{X^{C}, C>0 \}$, $ \{X^{(n)}, n\ge
1 \}$ and a stochastic process $X$ on the interval $[0,T]$. Let
$\mathbb{Q}^{n,C}$, $\mathbb{Q}^{C}$, $\mathbb{Q}^{n}$, and $\mathbb
{Q}$ be their corresponding measures. Suppose that, for any $C>0$,
$\mathbb{Q}^{n,C}\stackrel{W}{\longrightarrow}\mathbb{Q}^C$,
$n\rightarrow\infty$, and that $\mathbb{Q}^{C}\stackrel
{W}{\longrightarrow}\mathbb{Q}$ as $C\rightarrow\infty$. Suppose
further that, for any $\epsilon>0$,
\[
\lim\limits
_{C\rightarrow\infty}\overline{\lim\limits
_{n\rightarrow
\infty}}\pr \Bigl\{\sup\limits
_{0\le t \le T}
\bigl\llvert X_t^{(n,C)} - X_t^{(n)}\bigr
\rrvert \ge\epsilon \Bigr\}=0.
\]
Then $\mathbb{Q}^n\stackrel{W}{\longrightarrow}\mathbb{Q}$,
$n\rightarrow\infty$.
\end{thm}

Let $b,\si:\R\rightarrow\R$ be continuous functions. Consider the
stochastic differential equation
\begin{equation}
\label{dX} dX(t)=\si \bigl(X(t) \bigr)dW(t)+b \bigl(X(t) \bigr)dt,
\end{equation}
where $W=(W(t))$ is a Wiener process.
\begin{thm}\label{watanabe1}\cite[p.\ 177]{Watanabe}.
If $\si$ and $b$ are continuous functions satisfying the condition
\begin{equation}
\label{gr_cond}
\big|\si(x)\big|^2+\big|b(x)\big|^2\le K \bigl(1+|x|^2\bigr)
\end{equation}
for some positive constant $K$, then for any solution of \eqref{dX}
such that $\M (|X(0)|^2 )<\infty$, we have $\M
(|X(t)|^2 )<\infty$ for all $t>0$.
\end{thm}
\begin{thm}\label{watanabe2}\cite[p.\ 182]{Watanabe}.
Suppose that $\si$ and $b$ are bounded functions. Assume further that
the following conditions are satisfied:
\begin{itemize}
\item[\rm(i)] there exists a strictly increasing function $\rho(u)$ on
$ [0,\infty )$ such that
\[
\rho(0)=0, \qquad \int\limits_{0+}\rho^{-2}(u)du=\infty, \quad and \quad \big\llvert \si(x)-\si(y)\big\rrvert \le\rho\big(|x-y|\big)
\]
for all $x,y\in\R$.
\item[\rm(ii)] there exists an increasing and concave function $k(u)$
on $ [0,\infty )$ such that
\[
k(0)=0, \qquad \int\limits_{0+}k^{-1}(u)du=\infty, \quad and \quad \big\llvert b(x)-b(y)\big\rrvert \le k\big(|x-y|\big)
\]
for all $x,y\in\R.$
\end{itemize}
Then the pathwise uniqueness of solutions holds \xch{for}{for Eq.}~\eqref{dX}, and
hence it has a unique strong solution.
\end{thm}







%

\end{document}